\documentclass[a4paper,12pt]{amsart}

\usepackage{amssymb, enumitem}
\usepackage{color}

\usepackage[breaklinks=true,colorlinks=true,linkcolor=green,citecolor=red,urlcolor=blue]{hyperref}

\allowdisplaybreaks

\newcommand \al{\alpha}

\newcommand \on{\overline{\nabla}}
\newcommand \G{\Gamma}

\newcommand \la{\lambda}

\newcommand \br{\mathbb{R}}
\newcommand \bc{\mathbb{C}}
\newcommand \bh{\mathbb{H}}
\newcommand \bO{\mathbb{O}}

\newcommand \Span{\operatorname{Span}}
\newcommand \id{\operatorname{id}}
\newcommand \<{\langle}
\renewcommand \>{\rangle}
\newcommand \ip{\< \cdot, \cdot \>}

\newcommand \mS{\mathcal{S}}

\newcommand \mV{\mathcal{V}}
\newcommand \mM{\mathcal{M}}

\newcommand \Tr{\operatorname{Tr}}

\newcommand \tM{\overline{M}}

\newcommand \tR{\overline{R}}

\newcommand \gv{\mathfrak{v}}
\newcommand \gz{\mathfrak{z}}
\newcommand \ga{\mathfrak{a}}
\newcommand \gs{\mathfrak{s}}
\newcommand \gp{\mathfrak{p}}
\newcommand \gh{\mathfrak{h}}

\newcommand \rJ{\mathrm{J}}

\theoremstyle{plain}
\newtheorem{theorem}{Theorem}
\newtheorem*{theorem*}{Theorem}
\newtheorem*{corollary*}{Corollary}

\newtheorem*{conj*}{Conjecture}
\newtheorem{lemma}{Lemma}
\newtheorem{proposition}{Proposition}
\newtheorem*{prop*}{Proposition}

\theoremstyle{definition}

\newtheorem*{definition*}{Definition}

\theoremstyle{remark}
\newtheorem{remark}{Remark}

\begin{document}

\title{Einstein hypersurfaces of Damek-Ricci spaces}

\author{Sinhwi Kim}
\address{Department of Mathematics, Sungkyunkwan University, Suwon, 16419, Korea}
\email{kimsinhwi@skku.edu}

\author{Yuri Nikolayevsky}
\address{Department of Mathematics and Statistics, La Trobe University, Melbourne, Victoria, 3086, Australia}
\email{y.nikolayevsky@latrobe.edu.au}
% grant from which Jeong paid?

\author{JeongHyeong Park}
\address{Department of Mathematics, Sungkyunkwan University, Suwon, 16419, Korea}
\email{parkj@skku.edu}

\thanks {The second author was partially supported by ARC Discovery grant DP210100951.
The third author was supported by Samsung Science and Technology Foundation under Project Number SSTF-BA2001-03. }
%\thanks {The first and the third author were supported by the National Research Foundation of %Korea(NRF) grant funded by the Korea government(MSIT) (NRF-2019R1A2C1083957). \\
%\indent The second author was partially supported by ARC Discovery grant DP210100951.}

% 53C25  Special Riemannian manifolds (Einstein, Sasakian, etc.)
% 53B20  Local Riemannian geometry
% 53B25  Local submanifolds
% 53C30  Homogeneous manifolds
% 53C35  Symmetric spaces

%\date{}

\begin{abstract}
Einstein hypersurfaces are ``very rare" in rank-one symmetric spaces. Damek-Ricci spaces may be viewed as the closest and the most natural generalisations of noncompact rank-one symmetric spaces. We prove that no Damek-Ricci space admits an Einstein hypersurface.
\end{abstract}

\maketitle

\section{Introduction}
\label{s:intro}

The study of homogeneous manifolds is one of the main avenues in modern Riemannian geometry. In particular, the theory of homogeneous Einstein manifolds is a very active area, with many beautiful results and challenging conjectures. In comparison, the study of submanifolds of homogeneous spaces seems to be much less developed. In this paper, we investigate a question in the overlap of homogeneous geometry and submanifold geometry, the classification of Einstein hypersurfaces in Damek-Ricci spaces.

Einstein hypersurfaces in Riemannian manifolds are very rare (which is na\"{\i}vely suggested by a parameter count), to the extent that the correct question may be to classify (say homogeneous) spaces admitting an Einstein hypersurface, rather than to classify Einstein hypersurfaces in a given space or class of spaces. Nevertheless, an explicit classification of Einstein hypersurfaces, to the best of our knowledge, is only known for two-point homogeneous spaces. By \cite[Theorem~7.1]{Fia}, an Einstein hypersurface in a space of constant curvature is locally either totally umbilical, or totally geodesic, or is of conullity $1$, or is the product of two spheres of the same Ricci curvature in the sphere. The classification for rank-one symmetric spaces of non-constant curvature is summarised in the following theorem (where we assume that the metric is always normalised in such a way that the sectional curvature lies in $[\frac14, 1]$).

\begin{theorem}\label{th:rk1}
{\ }

  \begin{enumerate}[label=\emph{\arabic*.},ref=\arabic*]
    \item
    There are no Einstein (real) hypersurfaces in the complex projective space and in the complex hyperbolic space \cite[Theorem~4.3]{Kon}, \cite[Corollary~8.2]{Mon}, \cite[Theorem~8.69]{CR}.

    \item
    There are no Einstein (real) hypersurfaces in the quaternionic hyperbolic space \cite[Corollary~1]{OP}. A connected (real) hypersurface in $\mathbb{H}P^m, \, m \ge 2$, is Einstein if and only if it is a domain of a geodesic sphere of radius $r \in (0, \pi)$, where $\cos r = \frac{1-2m}{1+2m}$ \cite[Corollary~7.4]{MP}.

    \item
    There are no Einstein hypersurfaces in the Cayley hyperbolic plane. A connected hypersurface in the Cayley projective plane is Einstein if and only if it is a domain of a geodesic sphere of radius $r \in (0, \pi)$ such that $\cos r = -\frac{5}{11}$ \cite{KNP2}.
  \end{enumerate}
\end{theorem}

\begin{remark} \label{rem:err}
  An unfortunate typing error occurred in Example~1 of \cite{KNP2}. The correct multiplicities of the principal curvatures are $8$ for $\frac12 \cot \frac12 r$ and $7$ for $\cot r$ (not vice versa, as published). And then the equation in the last line of Example~1 should read $1+\cot^2 r - (7\cot r + 4 \cot \frac12 r) \cot r = \frac14 + \frac14 \cot^2 \frac12r - (7\cot r + 4 \cot \frac12 r) \frac12 \cot \frac12 r$. This does not affect the rest of the paper, including the value of $r_0$ given in Example~1 of \cite{KNP2}.
\end{remark}

In this paper, we study Einstein hypersurfaces of Damek-Ricci spaces. This class of spaces can be naturally considered as the closest generalisation of the class of noncompact rank-one symmetric spaces. A \emph{Damek-Ricci space} is a solvable Lie group with a left-invariant metric whose Lie algebra is obtained by a extending a generalised Heisenberg algebra $\mathfrak{n}$ by a derivation which acts as the identity on the centre $\gz$ of $\mathfrak{n}$ and as a $\frac12$ times the identity on $\gz^\perp$ (see Section~\ref{ss:DRspaces} for details). Any Damek-Ricci space has rank one (in the sense that any Cartan subalgebra of its Lie algebra is one-dimensional), and all noncompact rank-one symmetric spaces of non-constant curvature are Damek-Ricci. All Damek-Ricci spaces are Einstein solvmanifolds, and moreover, are \emph{harmonic} manifolds. One of many equivalent definitions states that a Riemannian space is harmonic if a punctured neighbourhood of any point admits a harmonic function which depends only on the distance to that point. Any \emph{homogeneous} harmonic space is either flat, or rank-one symmetric, or is a Damek-Ricci space by \cite{Sz,Heb}; whether or not there exist nonhomogeneous harmonic spaces is an open question (for the current state of knowledge see \cite{Kn} and references therein). For recent results and developments in geometry of Damek-Ricci spaces and their submanifolds we refer the reader to \cite{CMO, DD,Kol}.

By Theorem~\ref{th:rk1}, noncompact rank-one symmetric spaces of non-constant curvature admit no Einstein hypersurfaces. We prove that this is still the case for Damek-Ricci spaces.
\begin{theorem} \label{th:noEinDR}
  A Damek-Ricci space admits no Einstein hypersurfaces.
\end{theorem}
% say that non-flat (non-abelian), or self explanatory?

\section{Preliminaries}
\label{s:pre}

\subsection{Einstein condition}
\label{ss:Ein}

Let $\tM$ be an Einstein Riemannian manifold of dimension $n+1 \ge 3$ and let $M$ be an Einstein hypersurface of $\tM$. Denote $\ip$ the metric tensor on $\tM$ and the induced metric tensor on $M$. Denote $\on, \tR$ and $\nabla, R$ the Levi-Civita connection and the curvature tensor of $\tM$ and of $M$ respectively; for $x \in \tM$ and $X \in T_x\tM$, the Jacobi operator $\tR_X$ is defined by $\tR_XY=\tR(Y,X)X$ for $Y \in T_x\tM$. Let $\xi$ be a unit normal vector field of $M$. For $x \in M$ we define the shape operator $S$ on $T_xM$ by $SX=-\on_X \xi$, so that $\<\on_XY,\xi\>=\<SX,Y\>$, where $X, Y \in T_xM$. We will work on a small open, connected subset $\mM$ of $M$ on which the multiplicities of the eigenvalues of $S$ are constant. On $\mM$, we can choose a smooth orthonormal frame $X_i$ of eigenvectors of $S$, with the corresponding eigenvalues (principal curvatures) $\la_i, \; i=1, \dots, n$. Denote $H = \Tr S = \sum_{i=1}^n \la_i$ the mean curvature of $M$.

% maybe say that id without subscript when clear from context
By Gauss equation, $\tR(X_i,X_k,X_k,X_j) = R(X_i,X_k,X_k,X_j) + (\la_k^2 \delta_{ik}\delta_{jk}-\la_i \la_k\delta_{ij})$. Summing up by $k$ we obtain
  \begin{gather}\label{eq:Gauss}
    \<\tR_\xi X_i, X_j\>=\alpha_i \delta_{ij}, \quad \text{where } \\
    \alpha_i=-\la_i^2 + H \la_i + C, \label{eq:Gaussal}
  \end{gather}
and where $C$ is the difference of the Einstein constants of $\tM$ and of $M$, and $\alpha_i$'s are the eigenvalues of the restriction of $\tR_\xi$ to $T_xM$. Equivalently, \eqref{eq:Gauss} can be written as
  \begin{equation}\label{eq:Gaussinv}
    (\tR_\xi)_{|T_xM} = -S^2 + HS + C \id_{T_xM}.
  \end{equation}
For an eigenvalue $\al$ of the restriction of $\tR_\xi$ to $T_xM$, denote $L_{\al}$ the corresponding eigenspace. Then for every $\al$, the eigenspace $L_{\al}$ is $S$-invariant (hypersurfaces with this property are called \emph{curvature-adapted}) and
  \begin{equation}\label{eq:Gaussinval}
    S_{|L_{\al}}^2 - HS_{|L_{\al}} + (\al-C) \id_{L_{\al}} = 0.
  \end{equation}

Codazzi equation takes the form %separate?
  \begin{equation}\label{eq:codazzi}
    \tR(X_k,X_i,X_j,\xi)  = \delta_{ij} X_k(\la_j) - \delta_{kj} X_i(\la_j) + (\la_i-\la_j) \G_{ki}^{\hphantom{k}j} - (\la_k-\la_j) \G_{ik}^{\hphantom{i}j},
  \end{equation}
where $\G_{ki}^{\hphantom{k}j}=\<\nabla_k X_i, X_j\>$ (note that $\G_{kj}^{\hphantom{k}i}=-\G_{ki}^{\hphantom{k}j}$).

Differentiating equation~\eqref{eq:Gauss} in the direction of $X_k$ we obtain
\begin{gather}\label{eq:dG1nR}
  (\on_k \tR)(X_i,\xi,\xi,X_i) = X_k(\al_i)+2 \la_k \<\tR_{X_i}\xi, X_k\>,  \\
  \G_{ki}^{\hphantom{k}j} (\alpha_j - \alpha_i) = \la_k (\tR(\xi,X_j,X_i,X_k) + \tR(\xi,X_i,X_j,X_k)) -(\on_k \tR)(X_i,\xi,\xi,X_j),  \label{eq:dG2nR}
\end{gather}
where $i \ne j$ (here and below we abbreviate $\on_{X_k}$ to $\on_k$).

Let the ambient space $\tM$ be a Lie group with a left-invariant metric. For a vector field $T$ (defined at $x$), we denote $\widetilde{T}$ the left-invariant vector field on $\tM$ such that $\widetilde{T}(x)=T(x)$. Computing the covariant derivative of the curvature tensor for left-invariant vector fields we can write \eqref{eq:dG1nR} and \eqref{eq:dG2nR} in the following form, respectively:
\begin{gather}\label{eq:dG1}
\<\tR_{X_i}\xi, \on_k \widetilde{\xi} + \la_k X_k\>= -\tfrac12 X_k(\alpha_i), \\
\G_{ki}^{\hphantom{k}j} (\alpha_j - \alpha_i) = \<\tR(\xi,X_j)X_i + \tR(\xi,X_i)X_j,\on_k \widetilde{\xi} + \la_k X_k\> +  (\alpha_j - \alpha_i) \<\on_k \widetilde{X_i}, X_j\>, \label{eq:dG2}
\end{gather}
where $i \ne j$. Note that the vector $\on_k \widetilde{\xi} + \la_k X_k$ is the derivative of the \emph{Gauss map} of $M$ in the direction of $X_k$ \cite[Proposition~3]{Rip}.

\subsection{Damek-Ricci spaces}
\label{ss:DRspaces}

Let $(\mathfrak{n}, \ip)$ be a metric, two-step nilpotent Lie algebra with the centre $\gz$ and with $\gv=\gz^\perp$. For $Z \in \gz$, define $J_Z \in \mathfrak{so}(\gv)$ by $\<J_ZU,V\>=\<[U,V],Z\>$ for $U,V \in \gv$. The metric algebra $(\mathfrak{n}, \ip)$ is called a \emph{generalised Heisenberg algebra} if for all $Z \in \gz$, we have $J_Z^2=-\|Z\|^2 \id_{\gv}$. Note that $\gv$ is a Clifford module over the Clifford algebra $\mathrm{Cl}(\gz, -\ip_{\gz})$. Consider a one-dimensional extension $\gs=\mathfrak{n} \oplus \ga$ of a generalised Heisenberg algebra $\mathfrak{n}$, where $\ga = \br A$ and $[A,U]=\frac12 U, \; [A,Z]=Z$ for $U \in \gv, \; Z \in \gz$, and extend the inner product from $\mathfrak{n}$ to $\gs$ in such a way that $A \perp \mathfrak{n}$ and $\|A\|=1$. Then $\gs$ is a metric, solvable Lie algebra. The corresponding simply connected Lie group $\mS$ with the left-invariant metric defined by $\ip$ is called a \emph{Damek-Ricci space}. Note that the hyperbolic space can be obtained by a similar construction, starting with an abelian algebra $\mathfrak{n}$; it is conventional to exclude this case hence assuming that both $\gz$ and $\gv$ have positive dimension.

\begin{remark} \label{rem:dim}
We denote $d_\gz = \dim \gz$ and $d_\gv = \dim \gv$. From the representation theory of Clifford modules we know that if $d_\gv=2^{4a+b}c$, where $0 \le b \le 3$ and $c$ is odd, then $d_\gz \le 8a+2^b-1$.

% with $d_\gz, d_\gv > 0$
A Damek-Ricci space is symmetric in the following cases: $d_\gz=1$, $(d_\gz, d_\gv)=(7,8)$, and $(d_\gz, d_\gv)=(3,4m)$, when all the irreducible $4$-dimensional $\mathrm{Cl}(\gz)$-submodules of $\gv$ are isomorphic (that is, when $J_{Z_1} J_{Z_2} J_{Z_3} =\pm \id_{\gv}$ for an orthonormal basis $\{Z_1, Z_2, Z_3\}$ for $\gz$). The corresponding Damek-Ricci space is rank-one symmetric and is isometric to the complex hyperbolic space, the Cayley hyperbolic plane and the quaternionic hyperbolic space respectively.
\end{remark}

Let $T_1, T_2 \in T_x\mS$, with $T_1=V+Y+sA, \; T_2 = U+X+rA$, where $V,U \in \gv, \, Y, X \in \gz$ and $r,s \in \br$ (we identify $T_x\mS$ with $\gs$ via left translation). Then by \cite[\S 4.1.8, \S 4.1.6]{BTV}, for the Jacobi operator of $\mS$ at $x$ and the covariant derivative we have respectively
\begin{align}\label{eq:Jac} % maybe give a general formula for R?
  &
  \begin{aligned}
    \tR_{T_1}T_2 & = \tfrac34 J_XJ_YV + \tfrac34 J_{[U,V]}V + \tfrac34 r J_YV -\tfrac34 s J_XV -\tfrac14 \|T_1\|^2 U + (\tfrac34 \<X, Y\> + \tfrac14 \<T_1,T_2\>)V \\
     & -\tfrac34 [U, J_YV] + \tfrac34 s [U,V] - (\|T_1\|^2 - \tfrac34 \|V\|^2) X + \<T_1,T_2\>Y \\
     & +(\tfrac34 \<U,J_YV\> - r(\|T_1\|^2 - \tfrac34 \|V\|^2) + s (\<T_1,T_2\> - \tfrac34 \<U,V\>))A,
  \end{aligned}
  \\
  &\on_{T_1} \widetilde{T_2} = -\tfrac12 J_XV -\tfrac12 J_YU - \tfrac12 r V -\tfrac12 [U,V] - r Y + \tfrac12 \<U,V\>A + \<X,Y\>A, \label{eq:nabla}
\end{align}
where, as above, for a vector $T \in T_x\mS$ (or a vector field $T$ on a neighbourhood of $x$) we denote $\widetilde{T}$ the left-invariant vector field such that $\widetilde{T}(x)=T(x)$.

In a generalised Heisenberg Lie algebra we have the following identities:
\begin{equation*}
  [V,J_YV] = \|V\|^2 Y, \quad [V,J_YU] - [J_YV,U] = 2 \<U, V\> Y, \qquad \text{for } U, V \in \gv, \, Y \in \gz.
\end{equation*}

Following \cite[\S 3.1.12]{BTV}, for nonzero vectors $V \in \gv$ and $Y \in \gz$ we define the operator $K_{V,Y}$ on the subspace $Y^\perp \cap \gz$ by
\begin{equation}\label{eq:defK}
K_{V,Y}X =  \|V\|^{-2}\|Y\|^{-1} [V, J_XJ_YV].
\end{equation}
The operator $K_{V,Y}$ is skew-symmetric, with all the eigenvalues of $K_{V,Y}^2$ lying in $[-1, 0]$. Furthermore,
\begin{equation}\label{eq:K2-1}
  K_{V,Y}^2X = - X \; \Leftrightarrow \; J_XJ_YV=\|Y\|J_{K_{V,Y}X}V.
\end{equation}

% curvature tensor of HH here? copy and paste from the proof in prop2?
\begin{remark} \label{rem:tg}
  On several occasions, we will use the following argument. Let $\gh=\ga \oplus \gv' \oplus \gz'$ be a subalgebra of $\gs$ such that $J_{\gz'} \gv' \subset \gv'$, where $\gv' \subset \gv$ and $\gz' \subset \gz$. Then by \cite{Rou}, the corresponding subgroup of $\tM$ is totally geodesic and is a ``smaller" Damek-Ricci space or the real hyperbolic space (when $\gv'=0$ or $\gz'=0$). Note that by \cite{KNP1}, this construction gives ``almost all" totally geodesic submanifolds of Damek-Ricci spaces. In particular, $\gh$ is closed under $\tR$ and $\on \, \tR$ (that is, $\tR(\gh,\gh)\gh, \, (\on_{\gh} \tR)(\gh,\gh)\gh \subset \gh$) and also $\on_{\gh} \widetilde{T} \subset \gh$ for a left-invariant $\widetilde{T} \in \gh$.

  Moreover, if that totally geodesic subgroup is a symmetric space, then it is rank-one symmetric (see~Remark~\ref{rem:dim}). Then we additionally have $(\on_{\gh} \tR)(\gh,\gh)\gh=0$, and also the following ``duality" property: if unit vectors $T_1, T_2 \in \gh$ are such that $T_2$ is an eigenvector of the Jacobi operator $\tR_{T_1}$, then $T_1$ is an eigenvector of the Jacobi operator $\tR_{T_2}$, with the same eigenvalue.
\end{remark}

\section{Proof of Theorem~\ref{th:noEinDR}}
\label{s:proof}

Let $\tM$ be a Damek-Ricci space and $M$ be an Einstein hypersurface in $\tM$. We adopt the notation of Section~\ref{s:pre}. We will work on a small open, connected neighbourhood $\mM \subset M$, and from now on, will replace $M$ by $\mM$. Let $\xi =V+Y+sA$ be a unit normal vector field of $M$, where $V$ and $Y$ lie in the left-invariant subbundles $\gv$ and $\gz$ respectively, and $s$ is a real function on $M$. We can assume that on $M$, the shape operator $S$, the Jacobi operator $\tR_\xi$ and  the operator $K_{V,Y}^2$ (defined by \eqref{eq:defK}) have constant number of pairwise distinct eigenvalues (and then the multiplicities of corresponding eigenvalues are also constant).

We split the proof of Theorem~\ref{th:noEinDR} into two cases. In Section~\ref{ss:nongen} we consider the ``special" cases, when one of the components $V, Y$ or $sA$ of $\xi$ is locally zero. The proof in the ``general" case, when all three are locally nonzero, is given in Section~\ref{ss:gen}.

\subsection{Special cases}
\label{ss:nongen}
In this section, we consider the cases when one of the components $V, Y$ or $sA$ of the unit normal vector $\xi=V+Y+sA$ is zero at all points of $M$.

First suppose that $\xi$ has no $\gv$-component. Then $TM$ contains the left-invariant subbundle $\gv$, and so, by Frobenius Theorem, it also contains $[\gv, \gv]=\gz$. It follows that $TM = \gv \oplus \gz$, and so $M$ is a domain on a nilpotent group, the generalised Heisenberg group from which the Damek-Ricci space has been constructed (see Section~\ref{ss:DRspaces}). But the latter is never Einstein \cite[\S~3.1.7]{BTV} (and in general, a nilpotent group with a left-invariant metric can only be Einstein if it is abelian \cite[Theorem~2.4]{Mil}).

Now suppose that on $M$, the unit normal vector field $\xi$ has no $A$-component. Then $\xi=V+Y$ (and we may assume that $V\ne 0$) and $TM$ contains the left-invariant vector field $A$. Let $T=U+X+rA$ be a tangent vector field on $M$, where $U \in \gv, \; X \in \gz, \; r\in\mathbb{R}$. Then by \eqref{eq:nabla} we have $\<\on_T A,\xi\>=\<-\frac12 V - Y, T\>$, and so $SA=\frac12\|Y\|^2V-\frac12\|V\|^2Y$. Note that by \eqref{eq:nabla}, $\on_A \gv \subset \gv$ and $\on_A \gz \subset \gz$. It follows that $\<\on_A \xi,SA\>= \frac12\|Y\|^2\<\on_AV,V\>-\frac12\|V\|^2 \<\on_A Y,Y\>= \frac14 (\|Y\|^2A(\|V\|^2)-\|V\|^2A(\|Y\|^2))= \frac14 A(\|V\|^2)$. On the other hand, $\<\on_A \xi,SA\>= -\|SA\|^2=-\frac14 \|Y\|^2\|V\|^2$, and so $A(\|V\|^2) = -\|Y\|^2\|V\|^2$. Now from \eqref{eq:Gaussinv} we obtain $\<\tR_\xi A, A\>= -\|SA\|^2 + H \<SA,A\> + C=C-\frac14 \|Y\|^2\|V\|^2$, and so from \eqref{eq:Jac}, $C=\frac14 \|Y\|^2\|V\|^2-\frac14\|V\|^2-\|Y\|^2=-\frac14(2- \|V||^2)^2$. It follows that $\|V\|$ is a constant and so from the above, $Y=0$. Then $C=-\frac14$ and $\xi = V$, so $TM$ contains the left-invariant subbundle $\gz$. Let $Z \in \gz$ be a unit, left-invariant vector field and $T=U+X+rA$ be a tangent vector field on $M$, with $U \in \gv, \; X \in \gz \; r\in\mathbb{R}$. Then from \eqref{eq:nabla} we have $\<\on_T Z, \xi\>=\frac12\<J_ZV,U\>$, and so $SZ=\frac12 J_ZV$. But now from \eqref{eq:Gaussinv} we get $\<\tR_\xi Z, Z\>= -\|SZ\|^2 + H \<SZ,Z\> + C=-\frac12$, while from \eqref{eq:Jac}, $\<\tR_\xi Z, Z\> =-\frac14$, a contradiction.

The last case to consider is a little more involved. We have the following.
\begin{proposition} \label{p:noz}
In a Damek-Ricci space, there is no Einstein hypersurface whose normal vector field $\xi$ locally has no $\gz$-component.
\end{proposition}

\begin{proof}
By assumption, we have $\xi = V + s A$, and we may also assume that both $s$ and $V$ are locally nonzero.

% Q rename?
From~\cite[Theorem~4.2(v)]{BTV}, the restriction of the Jacobi operator $\tR_\xi$ to $T_xM$ has two eigenvalues, $-1$ and $-\frac{1}{4}$, with corresponding eigenspaces
\begin{equation}\label{eq:nozeigenJ}
\begin{gathered}
L_{-1} = \{sZ+J_Z V \, | \, Z \in \gz\}, \qquad L_{-\frac{1}{4}} = \{\|V\|^2 Z - s J_Z V \, | \, Z \in \gz\} \oplus \br Q \oplus \gp, \\
\text{where} \quad Q = \|V\|^{-1}(sV-\|V\|^2A) \quad \text{and} \quad \gp = \{U \in \gv \, | \, \<U, V\> = 0, \; [U, V] = 0\}.
\end{gathered}
\end{equation}

We have the following.
{
\begin{lemma} \label{l:nozeigen}
  The functions $s, \|V\|^2, H$ and all the eigenvalues of $S$ are constant and
\begin{equation}\label{eq:constants}
  s^2=2C+1, \qquad \|V\|^2=-2C, \qquad H= - C s^{-1}.
\end{equation}
The eigenspaces of $S$ and the corresponding eigenvalues are given by
\begin{equation}\label{eq:nozeigen}
\begin{alignedat}{2}
  \mV_- & = \{sZ+J_Z V \, | \, Z \in \gz\} = L_{-1} & \quad \quad & \rho_- = \tfrac{1+s^2}{2s}\\
  \mV_1 & = \{\|V\|^2 Z - s J_Z V \, | \, Z \in \gz\} \oplus \br Q \oplus \gp_1 & \quad \quad & \rho_1 = \tfrac{1}{2}s\\
  \mV_2 & = \gp_2  & \quad \quad & \rho_2 = \tfrac{1-2s^2}{2s},
\end{alignedat}
\end{equation}
where $\gp_1 \oplus \gp_2 = \gp$.
\end{lemma}
\begin{proof}
The tangent bundle $TM$ contains the left-invariant subbundle $\gz$. By~\eqref{eq:nabla}, for any nonzero $Z \in \gz$ and a tangent vector $T= U + X + r A$, where $U \in \gv, \, X \in \gz, \, r \in \br$, we have $\<\on_T Z, \xi\> = \frac12 \<J_Z V,U\> + s\<X,Z\> = \<\frac12 J_Z V + sZ, T\>$, and so $SZ=\frac12 J_Z V + sZ$. We have $Z=s(sZ+J_Z V)+(\|V\|^2 Z - s J_Z V)$, and so by~\eqref{eq:nozeigenJ}, $\tR_\xi Z= -s(sZ+J_Z V)-\frac14(\|V\|^2 Z - s J_Z V)=(-s^2-\frac14\|V\|^2) Z - \frac34 s J_Z V$. Acting by both sides of \eqref{eq:Gaussinv} on $Z$ we get $SJ_Z V = (\frac12\|V\|^2 + 2sH+2C) Z +(\frac12 s + H) J_Z V$. As $S$ is symmetric, from $\<SZ,J_ZV\>$ we obtain $H=-Cs^{-1}$, and so $SJ_Z V = \frac12(1-s^2) Z +(\frac12 s -Cs^{-1}) J_Z V$. It follows that the subspace $\Span(Z, J_ZV)$ is $S$-invariant. By~\eqref{eq:nozeigenJ}, it is also $\tR_\xi$-invariant, and the restriction of $\tR_\xi$ to it has two different eigenvalues. As any eigenvector of $S$ is an eigenvector of $\tR_\xi$ by~\eqref{eq:Gaussinv}, we obtain that both $sZ+J_Z V$ and $\|V\|^2 Z - s J_Z V$ are eigenvectors of $S$. We have $S(sZ+J_Z V)=\frac12(1+s^2)Z+(s-Cs^{-1})J_ZV$, and so $s^2=2C+1$ and the eigenvalue of $S$ corresponding to $sZ+J_Z V$ is $\rho_-=\frac{1+s^2}{2s}$; moreover, $L_{-1}$ is the eigenspace of $S$ with eigenvalue $\rho_-$. Furthermore, $\|V\|^2 Z - s J_Z V$ is an eigenvector of $S$, with eigenvalue $\rho_1=\frac12 s$. Then the subspace $L_{-\frac{1}{4}}$ is the orthogonal sum of two eigenspaces $\mV_1$ and $\mV_2$ with corresponding eigenvalues satisfying the equation $-\frac14=-\rho^2-Cs^{-1} \rho +C$ which gives $\rho_1$ as above and $\rho_2=\frac{1-2s^2}{2 s}$.

As $s^2=2C+1$, we obtain that $s, \|V\|, H$ and the eigenvalues of $S$ are constant.

It remains to show that $Q \in \mV_1$. We have $\<\on_Q \xi, Q\> = \|V\|^{-1}\<\on_Q V + s \on_Q A, sV-\|V\|^2A\>=\|V\|^{-1}(-\|V\|^2\<\on_Q V, A\>+s^2\<\on_Q A, V\>)=-\frac12 s$ from~\eqref{eq:nabla}. Then $\<SQ,Q\>=\frac12 s = \rho_1$, which implies $SQ=\rho_1Q$, as the restriction of the second fundamental quadratic form $\<ST,T\>$ to the unit sphere of $L_{-\frac{1}{4}}$ has two extremal values, $\rho_1$ and $\rho_2$.
\end{proof}
}

Next we need the following fact.

{
\begin{lemma}\label{l:nozd2}
  We have $d_2 > d_\gz + \frac12 d_\gv$, where $d_2=\dim \mV_2$. In particular, the subspace $\gp_2$ is nonzero and $\rho_1 \ne \rho_2$.
\end{lemma}
\begin{proof}
  We have $H=d_\gz \rho_- + d_2 \rho_2 + (d_\gv-d_2) \rho_1$, and so by~\eqref{eq:nozeigen} we obtain $(1+d_\gz+d_\gv-3d_2)s^2+(d_\gz+d_2-1)=0$. As $s \in (0, 1)$, we get $d_2 > \frac13(1+d_\gz+d_\gv)$ and $d_2 > d_\gz + \frac12 d_\gv$.
\end{proof}
}

We now take $X_k = P, \, X_j = P' \in \mV_2 \subset \gp$ and $X_i = sZ+J_ZV \in \mV_-$ in \eqref{eq:dG2}. Then $\alpha_j=-\frac14, \, \alpha_i = -1$ and $\la_k=\rho_2$. From~\eqref{eq:nabla} (and \eqref{eq:nozeigen}) we find
\begin{equation*}
  \on_k \widetilde{\xi} + \la_k X_k= \tfrac{1-3s^2}{2s} P, \qquad \<\on_k \widetilde{X_i}, X_j\>= -\tfrac12 s \<J_Z P, P'\>.
\end{equation*}
From~\cite[\S 4.1.7]{BTV} we obtain $\tR(\xi,P',P,X_i) = \tfrac14 \<J_ZP,P'\>$, and so by the first Bianchi identity, $\tR(\xi,X_i,P',P) = -\frac12\<J_ZP,P'\>$. Then~\eqref{eq:dG2} gives  $\G_{ki}^{\hphantom{k}j}=\frac{2s^2-1}{2s}\<J_ZP, P'\>$, and so from~\eqref{eq:codazzi} we get $(1-3s^2) \<J_ZP,P'\> = 0$, for all $P, P' \in \mV_2$ and $Z \in \gz$.

Now if $s^2=\frac13$, then from~\eqref{eq:nozeigen} we get $\rho_1 = \rho_2$ which contradicts Lemma~\ref{l:nozd2}. If $s^2 \ne \frac13$, then the subspace $\mV_2 \subset \gv$ is an isotropic subspace of $J_Z$, and so $d_2 \le \frac12 d_\gv$ which again contradicts Lemma~\ref{l:nozd2}.
\end{proof}

\subsection{General case}
\label{ss:gen}

In this section, we prove Theorem~\ref{th:noEinDR} assuming that on $M$, all three components of the unit normal vector field $\xi=V+Y+sA$ are nonzero (where $Y \in \gz, \, V \in \gv$). By Theorem~\ref{th:rk1} we can assume that $\tM$ is not symmetric.

Let $K=K_{V,Y}$ be the operator defined by \eqref{eq:defK}. By \cite[Theorem~4.2(vi)]{BTV}, the restriction of the Jacobi operator $\tR_\xi$ to $T_xM$ has eigenvalues $-1, -\frac14$, and (if $K^2 \ne -\id$), some eigenvalues lying in $(-1,0] \setminus \{-\frac14\}$.

The eigenspaces $L_{-1}$ and $L_{-\frac14}$ of $\tR_\xi$ are constructed as follows. We define the subspaces $\gs_4=\Span(A,V,Y,J_YV)$ and $\gp =\{U \in \gv \, | \, [U,V]=[U,J_YV]=0\} \subset \gv$, and the unit vector $T^0=\|Y\|^{-1}(J_YV+sY-\|Y\|^2A) \in \gs_4$. Furthermore, let $\gz_{-1} \subset \gz$ be the $(-1)$-eigenspace of $K^2$ and $\gv_{-1}=J_{\gz_{-1}} V$. Then $L_{-1} \oplus L_{-1/4} \oplus \br \xi = \gs_4 \oplus \gp \oplus \gz_{-1} \oplus \gv_{-1}$ and we have
\begin{align}\label{eq:L-1}
  L_{-1} &= \br T^0 \oplus \{(\|V\|^2-1)Z+J_{\|Y\|KZ-sZ}V\, | \, Z \in \gz_{-1}\}, \\
  L_{-1/4} &= (\gs_4 \cap \Span(\xi, T^0)^\perp) \oplus \gp \oplus \{\|V\|^2Z+J_{\|Y\|KZ-sZ}V\, | \, Z \in \gz_{-1}\}. \label{eq:L-1/4}
\end{align}
We note that the subspaces $\gz_{-1}, \gv_{-1}$ and $\gp$ can be trivial. Also note that $\gz_{-1}$ is $K$-invariant, and by \eqref{eq:K2-1}, $J_ZJ_YV=\|Y\|J_{KZ}V$, for $Z \in \gz_{-1}$ (and so also $J_{KZ}J_YV=-\|Y\|J_ZV$). The subspace $\gp$ is $J_Y$-invariant. We also have $\gv_{-1}=J_{\gz}V \cap J_{\gz}J_YV$ and $\gp = (J_{\gz} V + J_{\gz} J_YV)^\perp \cap \gv$, so for $d_\gp = \dim \gp$ and  $d_{-1}=\dim \gz_{-1}$ we get
\begin{equation}\label{eq:dimKhas-1}
d_\gv = d_\gp + 2 d_\gz - d_{-1}.
\end{equation}

We start with the following fact.

\begin{proposition}\label{p:-1}
  Let $M$ be an Einstein hypersurface of the Damek-Ricci space $\tM$. Suppose that all three components $V, Y$ and $sA$ of the unit normal vector field $\xi=V+Y+sA$ are locally nonzero. Then
  \begin{enumerate}[label=\emph{(\alph*)},ref=\alph*]
    \item \label{it:-1}
    The vector $T^0$ is an eigenvector of $S$.

    \item \label{it:mu}
    The operator $K^2$ has an eigenvalue different from $-1$.
  \end{enumerate}
\end{proposition}
% second is equiv to R_\xi having eigenvalues different from -1, -1/4
\begin{proof}
We first note that if no eigenvalue of $K^2$ equals $-1$, then there is nothing to prove. Indeed, assertion~\eqref{it:mu} follows immediately, as otherwise $d_\gz=1$, and so $\tM$ is a symmetric space isometric to the complex hyperbolic space (see Remark~\ref{rem:dim}). Assertion~\eqref{it:-1} also follows, as then $L_{-1} = \br T^0$ by \eqref{eq:L-1}. As $L_{-1}$ is $S$-invariant by \eqref{eq:Gaussinv}, $T^0$ is an eigenvector of $S$.

For the rest of the proof of the proposition we assume that $-1$ is an eigenvalue of $K^2$. Since $d_{-1}$ is even (as $\gz_{-1}$ is $K$-invariant), we have $d_{-1} \ge 2$ and so $d_\gz \ge 3$.

% maybe comment above that dG1 and 2 don't depend on Einstein condition...
By \eqref{eq:dG1} we have $\<\tR_T \xi, \on_k \widetilde{\xi} + \la_k X_k\>= 0$, for any $T \in L_{-\frac14}$. Furthermore, by Remark~\ref{rem:tg}, for any (nonzero) $P \in \gp$, the subspace $\gh=\gs_4 \oplus \Span(P, J_YP)$ is a subalgebra tangent to the totally geodesic subgroup $M' \subset \tM$ which is isometric to the complex hyperbolic space $\bc H^3$; then by the duality property, we get $\tR_T \xi = -\frac14 \|T\|^2 \xi$, for all $T \in L_{-1/4} \cap \gh$. A similar argument applies to the subspace $\gh=\gs_4 \oplus \Span(Z, KZ, J_ZV, J_{KZ}V)$; for a nonzero $Z \in \gz_{-1}$, it is tangent to the totally geodesic subgroup isometric to the quaternionic hyperbolic plane $\bh H^2$, and so we obtain $\tR_T \xi = -\frac14 \|T\|^2 \xi$, for all $T \in L_{-1/4} \cap \gh$. It follows from \eqref{eq:L-1/4} that we can take $T \in \gs_4^\perp \cap L_{-\frac14}$ in the equation $\<\tR_T \xi, \on_k \widetilde{\xi} + \la_k X_k\>= 0$. Let $T=\|V\|^2 Z + J_{\|Y\|KZ-sZ} V+P$, where $P \in \gp$ and $Z \in \gz_{-1}$. From \eqref{eq:Jac} we obtain $\tR_{T} \xi = \tfrac34 \|V\|^2 (J_YJ_Z+\|Y\|J_{KZ}) P - \tfrac14 \|T\|^2 \xi$, and so
\begin{equation*}
    0=\<(J_YJ_Z+\|Y\|J_{KZ}) P, \on_k \widetilde{\xi} + \la_k X_k\>=\<(J_YJ_Z+\|Y\|J_{KZ}) P, (N_\xi + S) X_k\>,
\end{equation*}
for all $k$, where $N_\xi$ is the \emph{Nomizu operator} on $T_x\tM$ defined by $N_\xi T' = \on_{T'} \widetilde{\xi}$, for $T' \in T_x\tM$. Then the latter equation gives $(N_\xi^t + S)(J_YJ_Z+\|Y\|J_{KZ}) P = 0$, for all $P \in \gp$ and all $Z \in \gz_{-1}$. Denote $W=(J_YJ_Z+\|Y\|J_{KZ}) P \in \gv$. Computing the Nomizu operator $N_\xi$ from \eqref{eq:nabla} we find $N_\xi^t W = \frac12 J_Y W - \frac12 s W - \frac12 [V,W]$, and so we obtain $SW=-\frac12 J_Y W + \frac12 s W + \frac12 [V,W]$. But $\gp$ is $J_Y$-invariant. Replacing $P$ by $J_Y P$ we obtain $SJ_YW=\frac12 \|Y\|^2 W + \frac12 s J_YW + \frac12 [V,J_YW]$. As the operator $S$ is symmetric, we get $\<SJ_YW,W\>=\frac12 \|Y\|^2 \|W\|^2=-\frac12 \|Y\|^2 \|W\|^2$, which implies $W=0$, that is,
\begin{equation} \label{eq:PKhas-1}
    (J_YJ_Z+\|Y\|J_{KZ}) P = 0, \quad \text{for all } P \in \gp, \; Z \in \gz_{-1}.
\end{equation}

For a nonzero $Z \in \gz_{-1}$, the symmetric operator $F_Z=\|Z\|^{-2}\|Y\|^{-1}J_YJ_ZJ_{KZ}$ on $\gv$ satisfies $F_Z^2=\id_{\gv}$. Denote $\mV_{\pm}$ the $(\pm 1)$-eigenspaces of $F_Z$ respectively. Then $\mV_{+}$ contains the vectors $V, J_YV, J_ZV, J_{KZ}V$, and also the subspace $\gp$, by \eqref{eq:PKhas-1}. If $d_\gz=3$ we have $\gz=\Span(Y,Z,KZ)$, and so $\gv=\Span(V, J_YV, J_ZV, J_{KZ}V) \oplus \gp = \mV_{+}$. But then $F_Z=\id_{\gv}$, and so by Remark~\ref{rem:dim}, the Damek-Ricci space $\tM$ is symmetric (and is isometric to the quaternionic hyperbolic space). It follows that $d_\gz > 3$. Then for any $X \in \gz, \, X \perp Y, Z, KZ$, the operator $F_Z$ anti-commutes with $J_X$, and so $\dim \mV_{+} = \frac12 d_\gv$, and moreover, $J_XV, J_YJ_XV \in \mV_{-}$. Thus $\gp = \mV_{+} \cap (\Span(V, J_YV, J_ZV, J_{KZ}V))^\perp$. It follows that $\gp$ is $J_Z$-invariant, for any $Z \in \gz_{-1}$ (and is $J_Y$-invariant, by definition), and that $d_\gp = \frac12 d_\gv-4$, which by \eqref{eq:dimKhas-1} gives
\begin{equation}\label{eq:Khas-1ineq}
2 d_\gz - \tfrac12 d_\gv -4 = d_{-1} \ge 2.
\end{equation}
From Remark~\ref{rem:dim} we find that this inequality is only possible in the following cases: $(d_\gz, d_\gv)= (5,8), (6,8), (7,8), (7, 16), (8,16)$. We can exclude the case $(d_\gz, d_\gv)=(7,8)$, as then by Remark~\ref{rem:dim} the Damek-Ricci space $\tM$ is isometric to the Cayley hyperbolic plane.

Furthermore, the case $(d_\gz, d_\gv)=(8,16)$ is also not possible, as the following argument shows. If $(d_\gz, d_\gv)=(8,16)$, the Clifford module $\gv$ can be identified with the octonion plane $\bO^2$, and $\gz$, with the algebra of octonions $\bO$. For $Z \in \gz =\bO$ and $W=(W_1,W_2)^t \in \gv = \bO^2$, we have $J_ZW=(ZW_2,-Z^*W_1)^t$, where the asterisk denotes the octonion conjugation. We can choose bases for $\bO$ and $\bO^2$ in such a way that $Y=\|Y\|1$. Let $V=(V_1,V_2)^t$. Then the fact that $d_{-1} > 0$ implies that there exist two orthonormal, unit octonions $Z,Z'$ such that $-ZV_1=Z'V_2$ and $ZV_2=Z'V_1$. It follows that $\|V_1\|=\|V_2\| \ne 0$ and that $V_1 \perp V_2$. From the first equation, $Z'=-\|V_2\|^{-2}(ZV_1)V_2^* = \|V_2\|^{-2} (ZV_2)V_1^*$, as $V_1 \perp V_2$. But then the second equation is automatically satisfied. Therefore we have $J_ZJ_YV=\|Y\|J_{\|V_2\|^{-2} (ZV_2)V_1^*}V$, for all unit imaginary octonions $Z$ such that the octonion $\|V_2\|^{-2} (ZV_2)V_1^*$ is unit, orthogonal to $Z$ and imaginary. The first two conditions follow from the fact that $\|V_1\|=\|V_2\|$ and $V_1 \perp V_2$; the third one gives $Z \perp V_1V_2^*$ (note that $V_1V_2^* \perp 1$). It follows that $J_ZJ_YV=\|Y\|J_{\|V_2\|^{-2} (ZV_2)V_1^*}V$ for all $Z$ in the six-dimensional subspace $(\Span(1,V_1V_2^*))^\perp \subset \bO$, and so $d_{-1}=6$ which violates the inequality~\eqref{eq:Khas-1ineq}.

We therefore have three cases to consider: $(d_\gz, d_\gv, d_{-1},d_\gp)= (5,8,2,0), (6,8,4,0), (7, 16, 2,4)$, where the values for $d_{-1}$ and $d_p$ are obtained from \eqref{eq:dimKhas-1}, \eqref{eq:Khas-1ineq} (note that an argument similar to the above shows that the case $(d_\gz, d_\gv, d_{-1},d_\gp)= (7, 16, 2,4)$ is also impossible if the two irreducible Clifford submodules of $\gv$ are non-isomorphic, but we will not need this fact in the rest of the proof).

As in all three cases, $d_{-1} < d_\gz-1$, the operator $K^2$ has at least one eigenvalue other than $-1$ which completes the proof of assertion~\eqref{it:mu}. It remains to prove assertion~\eqref{it:-1}.

% remark: up to here still works for curvature-adapted, never used Einstein condition

We first consider the first and the third case simultaneously; in both cases, $d_{-1}=2$. Let $Z$ be a nonzero vector from $\gz_{-1}$, so that $\gz_{-1}=\Span(Z, KZ)$. Consider the subspace $\gh=\gs_4 \oplus (\gz_{-1} \oplus \gv_{-1}) \oplus \gp$. If $(d_\gz, d_\gv, d_{-1},d_\gp)= (5,8,2,0)$, we have $\gp=0$, and so $\gh$ is tangent to a totally geodesic subgroup of $\tM$ isometric to $\bh H^2$. If $(d_\gz, d_\gv, d_{-1},d_\gp)= (7, 16, 2,4)$, the space $\gp$ is $J_{\gz_{-1}}$- and $J_Y$-invariant from the above, and so $\gp=\Span(P,J_YP,J_ZP,J_{KZ}P)$ for some (any) nonzero $P \in \gp$. It follows from \eqref{eq:PKhas-1} and Remark~\ref{rem:tg} that $\gh$ is tangent to a totally geodesic subgroup of $\tM$ isometric to $\bh H^3$.

By (\ref{eq:L-1},\ref{eq:L-1/4}) we have $L_{-1}, L_{-1/4} \subset \gh$ (in fact, $\gh=L_{-1} \oplus L_{-1/4} \oplus \br \xi$). In equation \eqref{eq:dG2}, take $X_i \in L_{-1}, \, X_j, X_k \in L_{-1/4}$. As $\gh$ is tangent to a totally geodesic symmetric space, we have $(\on_k \tR)(X_i,\xi,\xi,X_j)=0$, and so \eqref{eq:dG2nR} gives $\la_k \<\tR(\xi, X_i) X_j+\tR(\xi, X_j) X_i, X_k\> = \tfrac34\G_{ki}^{\hphantom{k}j}$. Moreover, by the duality property get $\tR_T \xi = -\frac14 \|T\|^2 \xi$, for all $T \in L_{-1/4}$, and so $\tR(\xi, X_j, X_k, X_i) +\tR(\xi, X_k, X_j, X_i)=0$. Then by the first Bianchi identity, $\<\tR(\xi, X_i) X_j+\tR(\xi, X_j) X_i, X_k\>= -3 \tR(\xi, X_j, X_k, X_i)$. We obtain
\begin{equation*}
\G_{kj}^{\hphantom{k}i} = 4 \la_k \tR(\xi, X_j, X_k, X_i).
\end{equation*}
Interchanging $j$ and $k$ and substituting into \eqref{eq:codazzi} we find %$\tR(X_k^\gh,X_j^\gh,X_i,\xi) + \tR(P_k,P_j,X_i,\xi) = ((\la_j-\la_i) \la_k + (\la_k-\la_i) \la_j)(4 \tR(\xi, X_j^\gh, X_k^\gh, X_i) - \<J_{a\|Y\|^{-1}Y+\|Y\|KZ_i-sZ_i}P_j,P_k\>) )$ which gives
\begin{equation}\label{eq:dG258716}
(\tfrac12+(\la_j-\la_i) \la_k + (\la_k-\la_i) \la_j) \tR(\xi, X_j, X_k, X_i) = 0.
\end{equation}
The term $\tR(\xi, X_j, X_k, X_i)$ can be computed from the curvature tensor of $\bh H^3$ (respectively of $\bh H^2$). Let $X_{i_1},X_{i_2},X_{i_3} \in L_{-1}$ be an orthonormal basis of eigenvectors of $S$. For every $i \in I=\{i_1,i_2,i_3\}$, there is a unique complex structure $J'_i$ on $\gh$ such that $J'_i \xi = X_i$ and that $\Span(J'_{i_1}, J'_{i_2}, J'_{i_3})$ is the quaternionic structure on $\gh$. The subspace $L_{-1/4}$ is $J'_i$-invariant for all $i \in I$; denote $\rJ_i$ the restriction of $J'_i$ to $L_{-1/4}$. Then we have $\rJ_{i_1} \rJ_{i_2} \rJ_{i_3}=\pm \id_{L_{-1/4}}$ and
$\tR(\xi, X_j, X_k, X_i)=-\frac14 \<J_i'X_j, X_k\>$.

Let $S'$ be the restriction of $S$ to $L_{-1/4}$ (recall that $L_{-1/4}$ is $S$-invariant by \eqref{eq:Gaussinv}). Then \eqref{eq:dG258716} gives $\<\rJ_i X_j, X_k\>+ 4 \<\rJ_i S'X_j, S'X_k\> - 2\la_i (\<\rJ_i S'X_j, X_k\>+\<\rJ_i X_j, S'X_k\>)=0$, for all $i \in I$. As this is satisfied for any $X_j, X_k \in L_{-1/4}$, we obtain \begin{equation}\label{eq:441}
\rJ_i+4S'\rJ_i S'-2\la_i(\rJ_i S'+S'\rJ_i)=0,
\end{equation}
for all $i \in I$.

We will show that all three $\la_i, \; i \in I$, are equal (then the claim of assertion~\eqref{it:-1} follows trivially). Assume they are not all equal. From \eqref{eq:Gaussinval} we have $S'^2-HS'-(C+\frac14)\id_{L_{-1/4}}=0$, and $\la_i^2-H\la_i-(C+1)=0$ for $i \in I$. If $S'=\tau \id$ for some $\tau \in \br$, then from \eqref{eq:441} we get $4\tau^2-4\la_i\tau+1=0$ contradicting our assumption. If $S'$ is not a multiple of the identity, but commutes with one of $\rJ_i, \, i \in I$, then from \eqref{eq:441} we obtain $\id+4{S'}^2-4\la_i S'=0$, and so $(2C+1)\id +2(H-\la_i)S'=0$. Then $C=-\frac12$ and $\la_i=H$, which contradicts the fact that $\la_i^2-H\la_i-(C+1)=0$. So we may assume that no $\rJ_i, \, i \in I$, commutes with $S$. Multiplying \eqref{eq:441} by $S'$ from the left we get $-(H\la_i+C) S' \rJ_i+(H \la_i + 3C) \rJ_i S' -(2C \la_j + H)\rJ_i=0$. Adding the transposed gives $(H \la_i+2C)(\rJ_i S'-S' \rJ_i)=0$, and so $H \la_i+2C=0$, for all $i \in I$. If $H \ne 0$, then all three $\la_i$'s are equal. If $H=0$, then $C=0$, and so $\la_i=\pm1$ and ${S'}^2=\frac14 \id$, so that the eigenvalues of $S'$ are $\pm\frac12$. For each $i \in I$, equation \eqref{eq:441} gives $(S'-\frac12 \la_i \id) \rJ_i (S'-\frac12 \la_i \id)=0$, and so the $(-\frac12 \la_i)$-eigenspace of $S'$ must be $\rJ_i$-isotropic. In particular, its dimension is at most $\frac12 \dim L_{-1/4}$. As not all the $\la_i$'s are the same, both $(\pm\frac12)$-eigenspaces of $S'$ have dimension $\frac12 \dim L_{-1/4}$. For each $i \in I$, one of these subspaces is $\rJ_i$-isotropic; but as $\rJ_i$ is orthogonal, the other subspace must also be $\rJ_i$-isotropic. It follows that $L_{-1/4}$ splits into the direct sum of two subspaces of dimension $\frac12 \dim L_{-1/4}$ which are isotropic relative to each of the operators $\rJ_i, \; i \in I$. Then each of $\rJ_i$ interchanges these subspaces which contradicts the fact that $\rJ_{i_1} \rJ_{i_2} \rJ_{i_3} = \pm \id$.

To complete the proof it remains to prove assertion~\eqref{it:-1} in the case $(d_\gz, d_\gv, d_{-1},d_\gp)= (6,8,4,0)$. In that case, $\gp$ is trivial. Moreover, for any nonzero $Z \in \gz_{-1}$, the subspace $\gh(Z)=\gs_4 \oplus (\Span(Z,KZ) \oplus \Span(J_ZV,J_{KZ}V))$ is tangent to a totally geodesic subgroup of $\tM$ isometric to $\bh H^2$ (note that $\gh(Z)=\gh(KZ)$).

In equation \eqref{eq:dG2}, take $X_i \in L_{-1}, \, X_j, X_k \in L_{-1/4}$. By \eqref{eq:L-1} we have $X_i=a_iT^0+(\|V\|^2-1)Z'+J_{\|Y\|KZ'-sZ'}V$ for some $a_i \in \br$ and $Z' \in \gz_{-1}$ such that $1=\|X_i\|^2=a_i^2+\|Z'\|^2(1-\|V\|^2)$. We may assume that $Z' \ne 0$ (otherwise $T^0$ is an eigenvector of $S$ and we are done). Let a nonzero $Z'' \in \gz_{-1}$ be such that $Z'' \perp Z', KZ'$. Note that $\gz_{-1}=\Span(Z', KZ',Z'',KZ'')$, as $d_{-1}=4$. Let $X_j=X_j' + X_j'', \; X_k=X_k' + X_k''$, where $X_j', X_k' \in \gh(Z')$ and $X_j'', X_k'' \perp \gh(Z')$. Note that $X_j'', X_k'' \in \gh(Z'')$ and that $X_j', X_k',X_j'',X_k'' \in L_{-1/4}$. We have $(\on_k \tR)(X_i,\xi,\xi,X_j) = (\on_{X_k'} \tR)(X_i,\xi,\xi,X_j)+(\on_{X_k''} \tR)(X_i,\xi,\xi,X_j')+(\on_{X_k''} \tR)(X_i,\xi,\xi,X_j'')$. But $(\on_{X_k'} \tR)(X_i,\xi)\xi=0$, as $X_k', X_i,\xi \in \gh(Z')$ which is tangent to a totally geodesic symmetric space. A similar argument applied to $X_i,\xi,X_j' \in \gh(Z')$ (and the second Bianchi identity) shows that $(\on_{X_k''} \tR)(X_i,\xi,\xi,X_j')=0$. And then $(\on_{X_k''} \tR)(X_i,\xi,\xi,X_j'')=0$, as $X_k'',\xi,X_j'' \in \gh(Z'')$. Then \eqref{eq:dG2nR} gives $\G_{kj}^{\hphantom{k}i} = -\tfrac43 \la_k (\tR(\xi, X_i, X_j, X_k) + \tR(\xi, X_j, X_i, X_k))$. Now given any $T \in L_{-1/4}$ we can find $Z_T \in \gz_{-1}$ such that $T \in \gh(Z_T)$. By the duality property (see Remark~\ref{rem:tg}), we have $\tR_T \xi = -\frac14 \|T\|^2 \xi$. Taking $T=X_j+X_k$ and polarising we obtain  $\tR(\xi, X_j, X_k, X_i) +\tR(\xi, X_k, X_j, X_i)=0$, and so by the first Bianchi identity, $\tR(\xi, X_j, X_i, X_k) = \frac12 \tR(\xi, X_i, X_j, X_k)$. Therefore $\G_{kj}^{\hphantom{k}i} = -2 \la_k \tR(\xi, X_i, X_j, X_k)$. Interchanging $j$ and $k$ and substituting into \eqref{eq:codazzi} we obtain
\begin{equation}\label{eq:dG268}
(\tfrac12+(\la_j-\la_i) \la_k + (\la_k-\la_i) \la_j) \tR(\xi, X_i, X_j, X_k) = 0.
\end{equation}
As above, we denote $S'$ the restriction of $S$ to $L_{-1/4}$ (recall that $\dim L_{-1/4}=6$). The operator $S'$ is symmetric and satisfies the equation $S'^2-HS'-(C+\frac14)\id_{L_{-1/4}} = 0$ by \eqref{eq:Gaussinval}. For a (unit) eigenvector $X_i \in L_{-1}$ of $S$ with corresponding eigenvalue $\la_i$ we denote $\rJ_i$ the skew-symmetric operator on $L_{-1/4}$ defined by $\<\rJ_i T_1,T_2\>=2 \tR(\xi, X_i, T_1, T_2)$, for $T_1, T_2 \in L_{-1/4}$. Then~\eqref{eq:dG268} takes the form
\begin{equation}\label{eq:441'}
\rJ_i+4S'\rJ_i S'-2\la_i(\rJ_i S'+S'\rJ_i)=0
\end{equation}
similar to \eqref{eq:441}. But the structure of the $\rJ_i$'s in this case is more complicated as in the previous case. We have $\tR(\xi, X_i, X_j, X_k) = \tR(\xi, X_i, X_j'+X_j'', X_k'+X_k'') = \tR(\xi, X_i, X_j', X_k') + \tR(\xi, X_i, X_j'', X_k'')$ (the other two terms are zeros, as $\xi, X_i,X_j',X_k' \in \gh(Z') \perp X_J'',X_k''$). Then, similar to the previous case, $\tR(\xi, X_i, X_j', X_k')=\frac{1}{2}\<\rJ_i'X_j', X_k'\>$, where $\rJ_i'$ is the restriction of one of the complex structures (belonging to the quaternionic structure on $\gh(Z')$ and uniquely defined by the fact that it maps $\xi$ to $X_i$) to the $4$-dimensional subspace $L_{-1/4} \cap \gh(Z')$. Note that ${\rJ_i'}^2=-\id$ on that subspace. To compute $\tR(\xi, X_i, X_j'', X_k'')$ we decompose $X_i$ as $X_i=a_iT^0+X_i'$, where $X_i'=(\|V\|^2-1)Z'+J_{\|Y\|KZ'-sZ'}V$, as above. Note that $\tR(\xi, X_i', X_j'', X_k'')=0$, as $\xi, X_j'', X_k'' \in \gh(Z'') \perp X_i'$. So $\tR(\xi, X_i, X_j'', X_k'')=a_i\tR(\xi, T^0, X_j'', X_k'')=\frac{1}{2}a_i\<J^0 X_j'', X_k''\>$, where $J^0$ is a skew-symmetric operator on the $2$-dimensional subspace $L_{-1/4} \cap (\gh(Z'))^\perp$ such that $(J^0)^2=-\id$ on that space. Thus, relative to the orthogonal decomposition $L_{-1/4}=(L_{-1/4} \cap (\gh(Z')) \oplus (L_{-1/4} \cap (\gh(Z'))^\perp)$, the matrix of $\rJ_i$ has the form $\rJ_i' \oplus \left(\begin{smallmatrix} 0 & a_i \\ -a_i & 0 \end{smallmatrix}\right)$ (but note that this decomposition itself depends on $X_i$). We note that $\rJ_i$ is nonsingular if and only if $a_i \ne 0$ (and has rank $4$ otherwise).

We now analyse equation \eqref{eq:441'} in several possible cases. Note that if all the eigenvalues $\la_i$'s of the restriction of $S$ on $L_{-1}$ are equal, there is nothing to prove, so we will assume that they are not. Then this restriction has two different eigenvalues which both satisfy the equation $\la_i^2-H\la_i-(C+1)=0$, by \eqref{eq:Gaussal}. The $5$-dimensional space $L_{-1}$ splits into orthogonal sum of corresponding eigenspaces. Note that if for all $X_i$ in one of these eigenspaces we have $a_i=0$, then the other eigenspace contains $T^0$ and we are done. Otherwise, we can choose an orthonormal basis $\{X_i\},\; i \in I=\{i_1,i_2,i_3,i_4,i_5\}$, for $L_{-1}$ such that $a_i \ne 0$ for all $i \in I$, and so the operators $\rJ_i$ in \eqref{eq:441'} are nonsingular.

Similar to the previous case, if $S'=\tau \id$ for some $\tau \in \br$, then from \eqref{eq:441'} we get $4\tau^2-4\la_i\tau+1=0$ which contradicts the assumption that not all $\la_i, \, i \in I$, are equal. If $S'$ is not a multiple of the identity, but commutes with one of $\rJ_i, \, i \in I$, then \eqref{eq:441'} gives  $\id+4{S'}^2-4\la_i S'=0$ (as $\rJ_i$ is nonsingular), and so $(2C+1)\id +2(H-\la_i)S'=0$. Then $C=-\frac12$ and $\la_i=H$, which contradicts the fact that $\la_i^2-H\la_i-(C+1)=0$. We may therefore assume that no $\rJ_i, \, i \in I$, commutes with $S$. Multiplying \eqref{eq:441'} by $S'$ from the left and adding the transposed we obtain $(H \la_i+2C)(\rJ_i S'-S' \rJ_i)=0$, and so $H \la_i+2C=0$, for all $i \in I$. As not all the $\la_i$'s are equal, we get $H=C=0$, and so $\la_i=\pm1$ and ${S'}^2=\frac14 \id$, so that the eigenvalues of $S'$ are $\pm\frac12$. Then for each $i \in I$, equation \eqref{eq:441'} gives $(S'-\frac12 \la_i \id) \rJ_i (S'-\frac12 \la_i \id)=0$, and so the $(-\frac12 \la_i)$-eigenspace of $S'$ must be $\rJ_i$-isotropic. Since $\rJ_i$ is nonsingular, the dimension of that eigenspace is at most $3=\frac12 \dim L_{-1/4}$. As not all the $\la_i$'s are the same, both $(\pm\frac12)$-eigenspaces of $S'$ have dimension $3$. But as the operators $\rJ_i$ are not orthogonal, we need an argument different from the one above to get a contradiction.

Recall that $H= \Tr S = \sum_{l=1}^{14} \la_l$. Out of $14$ eigenvalues $\la_l$ of $S$, we have five eigenvalues of the restriction of $S$ to $L_{-1}$, each of which being $\pm1$, and two eigenvalues $\pm \frac12$, each of multiplicity $3$, which are the eigenvalues of $S'$, the restriction of $S$ to $L_{-1/4}$. As $H=0$, the sum of the remaining three eigenvalues (which we label $\la_1, \la_2, \la_3$) must be an odd integer. They are constructed as follows. The skew-symmetric operator $K$ on the $5$-dimensional space $\gz \cap Y^\perp$ has a $1$-dimensional kernel. Let $Z_0$ be a unit vector in that kernel; note that its orthogonal complement in $\gz \cap Y^\perp$ is precisely $\gz_{-1}$. According to \cite[Theorem~4.2(vi)]{BTV}, the $3$-dimensional subspace $\Span(Z_0, J_{Z_0}V, J_{Z_0}J_YV)$ is $\tR_\xi$-invariant, and the restriction of $\tR_\xi$ to it has three pairwise different eigenvalues $\al_i, \; i=1,2,3$, which are the roots of the equation $f(t)=0$, where $f(t)=t^3+\frac32 t^2 + \frac{9}{16}t + q^2$ and $q^2 = \frac{1}{16}- \frac{27}{64} \|V\|^4 \|Y\|^2$ (it is easy to see that the right-hand side of the latter equation is always positive, and so we can take $q \in (0, \frac14)$). Up to relabelling, one has $-1<\al_1<-\frac34<\al_2<-\frac14<\al_3\le 0$. As $C=H=0$ in our case, \eqref{eq:Gaussal} gives $\la_i=\pm \sqrt{-\al_i}$, for $i=1,2,3$, and so
$0 \le |\la_3| < \frac12 < |\la_2| < \frac{\sqrt{3}}{2} < |\la_1|<1$. Clearly, $|\la_1|+|\la_2|+|\la_3|<3$, so if $\la_1+\la_2+\la_3$ is an odd integer, we must have $\la_1+\la_2+\la_3=1$ (up to changing the sign of $\xi$). Then from the same inequalities it follows that we must have $\la_1=\sqrt{-\al_1}$ and either $\la_2=\sqrt{-\al_2},\, \la_3=-\sqrt{-\al_3}$, or $\la_2=-\sqrt{-\al_2}, \, \la_3=\sqrt{-\al_3}$. In the first case, we get $\sqrt{-\al_1}+\sqrt{-\al_2}=1+\sqrt{-\al_3}$. Squaring both sides and using the fact that $\al_1+\al_2+\al_3=-\frac32$ and $\al_1\al_2\al_3=-q^2$ we obtain $\al_3+\frac14=-(-\al_3)^{-1/2}(q+\al_3)$. The left-hand side is positive, but the right-hand side is negative. To see that, we note that $f(0)=q^2>0$, but $f(-q)=-q(q-\frac94)(q-\frac14)<0$ (as $q \in (0, \frac14)$). So $\al_3$, the biggest root of $f$ lies in $(-q,0)$. In the second case, we have $\sqrt{-\al_1}+\sqrt{-\al_3}=1+\sqrt{-\al_2}$, and so by a similar calculation, $\al_2+\frac14=-(-\al_2)^{-1/2}(q+\al_2)$. Squaring both sides we get
$\al_2^3+\frac32\al_2^2+(2q+\frac{1}{16})\al_2+q^2=0$. As $f(\al_2)=0$, we get $2(q-\frac14)\al_2=0$, a contradiction.
\end{proof}

With Proposition~\ref{p:-1}, we can now complete the proof of Theorem~\ref{th:noEinDR}.

By~\cite[Theorem~4.2(vi)]{BTV}, the eigenspaces of $\tR_\xi$ with corresponding eigenvalues different from $-1$ and $-\frac14$ are constructed as follows. For each eigenvalue $\mu \ne -1$ of $K^2$, we consider the corresponding eigenspace $\gz_\mu \subset \gz \cap Y^\perp$ and the subspace $\gv_\mu=\Span(J_ZV,J_{KZ}V,J_ZJ_YV \, | \, Z \in \gz_\mu)$. Note that $\gz_\mu$ is $K$-invariant (and is of even dimension if $\mu \ne 0$).
Each of the (pairwise orthogonal) subspaces $\gz_\mu \oplus \gv_\mu$ splits into the orthogonal sum of three eigenspaces $L_{\alpha_{\mu,l}},\, l=1,2,3$, of $\tR_\xi$. The eigenvalues $\alpha_{\mu,l}$ satisfy the equation $(\alpha_{\mu,l}+1)(\alpha_{\mu,l}+\frac14)^2=\frac{27}{64}\|V\|^4\|Y\|^2(1+\mu)$; they are pairwise nonequal and lie in $(-1,0]\setminus\{-\frac14\}$.

For the rest of the proof, we choose and fix a particular eigenvalue $\mu \in (-1,0]$ of $K^2$ (note that at least one such eigenvalue exists by Proposition~\ref{p:-1}\eqref{it:mu}). To simplify notation, we will drop $\mu$ from the subscripts, so that we will write $\alpha_l, \, l=1,2,3$, for the corresponding $\alpha_{\mu,l}$.

The eigenspaces $L_{\alpha_{l}}, \, l=1,2,3$, are given by
\begin{gather*}
  L_{\alpha_{l}}= \{\eta_{l} \nu_{l} Z + 3\nu_{l} J_ZJ_YV-9\|V\|^2\|Y\|J_{KZ}V-3s\eta_{l} J_ZV \, | \, Z \in \gz_\mu\}, \\
  \text{where} \quad \eta_{l}=4\alpha_{l}+1, \, \nu_{l}= \eta_{l}+3\|V\|^2.
\end{gather*}

According to \eqref{eq:Gaussinv}, each of the eigenspaces $L_{\alpha_l},\, l=1,2,3$, is $S$-invariant. For each $l=1,2,3$, there is a linear bijection $\psi_l$ between $\gz_\mu$ and $L_{\al_l}$ which send $Z \in \gz_\mu$ to the vector $\psi_l(Z)=\eta_l \nu_l Z + 3\nu_l J_ZJ_YV-9\|V\|^2\|Y\|J_{KZ}V-3s\eta_l J_ZV \in L_{\al_l}$. Moreover, $\psi_l$ is a homothety as $\|\psi_l(Z)\|^2= \|Z\|^2 ((\eta_l \nu_l)^2 + (3\nu_l)^2 \|Y\|^2\|V\|^2 -81 \mu \|V\|^6\|Y\|^2 +(3s\eta_l)^2 \|V\|^2) - 54 \nu_l \|V\|^2 \|Y\| \<J_ZJ_YV,J_{KZ}V\>$. But $\<J_ZJ_YV,J_{KZ}V\>=\<KZ,[V,J_ZJ_YV]\>=\|Y\|\|V\|^2 \|KZ\|^2= - \mu \|Y\|\|V\|^2 \|Z\|^2$ by \eqref{eq:defK}. It follows that for each $l=1,2,3$, the operator $S_l$ on $\gz_\mu$ defined by $\psi_l(S_lZ)=S(\psi_l(Z))$ is symmetric. Moreover, we have $-S_l^2+H S_l +(C-\al_l)\id_{\gz_\mu}=0$ by \eqref{eq:Gaussinval}.

By Proposition~\ref{p:-1}\eqref{it:-1}, $T^0 \in L_{-1}$ is an eigenvector of $S$. Denote $\la_1$ the corresponding eigenvalue. We take $X_k = T^0$ and $X_i =\psi_i(Z_i) \in L_{\al_i}, \, X_j =\psi_j(Z_j) \in L_{\al_j}$, where $\al_i \ne \al_j$, in \eqref{eq:dG2}. Using \cite[\S~4.1.7]{BTV} and \eqref{eq:nabla} we compute $\G_{ki}^{\hphantom{k}j}$. Now taking the same $X_k, X_i$ and $X_j$ in the equation obtained from \eqref{eq:dG2} by first interchanging $i$ and $k$ we can similarly compute $\G_{ik}^{\hphantom{i}j}$. Substitute the expressions for $\G_{ki}^{\hphantom{k}j}$ and $\G_{ik}^{\hphantom{i}j}$ in \eqref{eq:codazzi} and multiply both sides by $\eta_j (=4\al_j+1)$. Adding to the resulting equation the same equation with $i$ and $j$ interchanged we arrive (after some computer assisted, but straightforward calculation) to the equation
\begin{equation}\label{eq:31zmunot0}
  m_1 \<KZ_i,Z_j\> + m_2 (\la_i-\la_j) \<KZ_i,Z_j\> + m_3 \<Z_i,Z_j\> + m_4 (\la_i-\la_j) \<Z_i,Z_j\> = 0,
\end{equation}
with the coefficients $m_a, \, a=1,2,3,4$, given by
\begin{equation}\label{eq:m_a} % divided by 6 \|V\|^2
\begin{split}
  m_1 =& 9 \|V\|^2 \|Y\|(\eta_i-\eta_j) (q -\sigma_2),\\
  m_2 =& 36 \|V\|^2 \|Y\|(s-\la_1) (\sigma_2 (\sigma_1+6)-2q),\\
  m_3 =& s (3\|V\|^2(2 \sigma_2^2 +3 \sigma_2 \sigma_1) + q( 2\sigma_1^2 - 8 \sigma_2 - 3 \|V\|^2 \sigma_1)),\\
  m_4 =& 2 (\eta_i-\eta_j) (9\|V\|^2 (s^2-\|Y\|^2-2s\la_1)\sigma_2 - q (12s^2+3\|V\|^2-12s\la_1+\sigma_1)),
\end{split}
\end{equation}
where $\sigma_1 = \eta_i + \eta_j, \; \sigma_2=\eta_i \eta_j$ and $q=27\|V\|^4\|Y\|^2(1+\mu)$. Note that computationally, it is easier to work with $\eta_i$ than with $\al_i$. The numbers $\eta_1, \eta_2, \eta_3$ are three pairwise different roots of the polynomial $p(t)=t^2(t+3)-q$; the expressions for $m_a$ in \eqref{eq:m_a} are reduced modulo $p$. Note that $\sigma_1=-3-\eta_k, \, \sigma_2=\eta_k (\eta_k+3)$, where $\{i,j,k\}=\{1,2,3\}$.

%old:
%\begin{equation}\label{eq:m_a}
%\begin{split}
%  m_1 =& 18 \|V\|^2 \|Y\|(\eta_i-\eta_j) (27 \|V\|^4 \|Y\|^2 (1+\mu) -\sigma_2),\\
%  m_2 =& -72 \|V\|^2 \|Y\|(s-\la_1) (54 \|V\|^4 \|Y\|^2 (1+\mu) -\sigma_2 (\sigma_1+6)),\\
%  m_3 =& -6 \|V\|^2 s (-2 \sigma_2^2 -3 \sigma_2 \sigma_1 + 9 \|V\|^2 \|Y\|^2 (1+\mu)( -2\sigma_1^2 + 8 \sigma_2 + 3 \|V\|^2 \sigma_1)),\\
%  m_4 =& 36 \|V\|^2 (\eta_i-\eta_j) ((1-\|V\|^2-2\|Y\|^2-2s\la_1)\sigma_2 - 3 \|V\|^2 \|Y\|^2 (1+\mu) \sigma_1 \\
%       & \quad + 9 \|V\|^2 \|Y\|^2 (1+\mu) (-4+3\|V\|^2+4\|Y\|^2+4s\la_1)),
%\end{split}
%\end{equation}

Equation \eqref{eq:31zmunot0} gives $\<(m_1 K_\mu + m_2 (K_\mu S_i-S_jK_\mu) + m_3 \id_{\gz_\mu} + m_4 (S_i-S_j))Z_i,Z_j\> = 0$, where $K_\mu$ is the restriction of $K$ to $\gz_\mu$. As both $\psi_i$ and $\psi_j$ are linear bijections, this implies $m_1 K_\mu + m_2 (K_\mu S_i-S_jK_\mu) + m_3 \id_{\gz_\mu} + m_4 (S_i-S_j)=0$ which is equivalent to $S_j(m_2 K_\mu + m_4 \id_{\gz_\mu})=m_1K_\mu+m_2K_\mu S_i+m_3\id_{\gz_\mu} + m_4S_i$. Multiplying both sides from the left by $m_2K_\mu-m_4\id_{\gz_\mu}$ we obtain a symmetric operator on the left-hand side, while on the right-hand side we get
$(m_1m_2\mu-m_3m_4)\id_{\gz_\mu} + (m_2^2 \mu -m_4^2)S_i+(m_2m_3-m_1m_4)K_\mu$. Suppose that $K_\mu \ne 0$. From \eqref{eq:m_a}, $m_2m_3-m_1m_4=0$ is a symmetric polynomial in $\eta_i, \eta_j$. Using the fact that $\sigma_1=-3-\eta_k$ and $\sigma_2=\eta_k (\eta_k+3)$, where $\{i,j,k\}=\{1,2,3\}$, and reducing modulo $p(\eta_k)=0$ we obtain $0=m_2m_3-m_1m_4=A_2\eta_k^2+A_1\eta_k + A_0$, where the coefficients $A_2, A_1, A_0$ depend only on $\|V||,\|Y\|,s, \mu$ and $A_2=q (1-3\|V\|^2) + 9(1+5\|V\|^2)(1 - \|V\|^2)$. As this is satisfied for all $k=1,2,3$ and as $\eta_1, \eta_2, \eta_3$ are pairwise different, we must have $A_2=A_1=A_0$. But $A_2 > 0$. This is obvious if $3\|V\|^2 \le 1$, and if $3\|V\|^2 > 1$, we have $q < 27 \|V\|^4(1-\|V\|^2)$, and so $A_2 >  9 (1-\|V\|^2)^2 (1+3\|V\|^2)^2>0$, a contradiction.

It remains to consider the case $K_\mu=0$. Then also $\mu = 0$. Using \eqref{eq:m_a} and the fact that $\sigma_1=-3-\eta_k, \, \sigma_2=\eta_k (\eta_k+3)$ we can write equation \eqref{eq:31zmunot0} in the following form (after multiplying by $\eta_k$ and reducing modulo $p(\eta_k)=0$):
\begin{equation}\label{eq:31zmu0}
\Phi(\eta_k) \, \id_{\gz_\mu} + (\eta_i-\eta_j) \, \Psi(\eta_k)\,  (S_i-S_j) = 0,
\end{equation}
where
\begin{equation}\label{eq:PhiPsi}
\begin{split}
\Phi(t) &= 3s ((3 \|V\|^2 + 2) t^2 + 6 (\|V\|^2 + 1) t - (9 \|V\|^2 + 2 q))\\
\Psi(t) &= 2 t^2 + 6(\|Y\|^2 -3 s^2 + 4 s \la_1) t + 18 \|V\|^2 (s^2 - 2 s \la_1 - \|Y\|^2).
\end{split}
\end{equation}
% old:
%\begin{equation}\label{eq:PhiPsi}
%\begin{split}
%\Phi(t) &= - s ((54 \|V\|^2\|Y\|^2+9) t^2 + (-81 \|V\|^4 \|Y\|^2+108 \|V\|^2\|Y\|^2+27) t\\
%&- 162 \|V\|^2\|Y\|^2 (\|V\|^2+1)), \\
%\Psi(t) &= 6 (-2 s \la_1-\|V\|^2-2\|Y\|^2+1) t^2+ 6(3 \|V\|^2\|Y\|^2-6 s\la_1\\
%&-3 \|V\|^2-6 \|Y\|^2+3)t +54 \|V\|^2\|Y\|^2 (4 s\la_1+3 \|V\|^2+4 \|Y\|^2-3).
%\end{split}
%\end{equation}
From \eqref{eq:31zmu0} we obtain that the cyclic sum of the expressions $(\eta_j-\eta_k) (\eta_k-\eta_i) \Psi(\eta_i) \Psi(\eta_j) \Phi(\eta_k)$ by $i,j,k$ is zero. As $\eta_i, \eta_j, \eta_k$ are three pairwise different roots of the equation $p(t) =0$, we obtain
\begin{equation*}
(q - 4)(s(3 \|V\|^2 -2) \la_1 + 2s^2(1-\|V\|^2))=0.
\end{equation*}
Note that $q - 4 < 0$ (as $s \ne 0$) and also $\|V\|^2 \ne \frac23$ (as otherwise $s=0$). We find $\la_1= \frac{2s(1-\|V\|^2)} {2 - 3 \|V\|^2}$. Furthermore, if $\Psi(\eta_k)=0$, then $\Phi(\eta_k)=0$ from \eqref{eq:31zmu0}, but this cannot be satisfied simultaneously with $p(\eta_k)=0$. It follows that $\Psi(\eta_k)\ne 0$, and so from \eqref{eq:31zmu0}, $S_i$ and $S_j$ commute. Multiplying both sides of \eqref{eq:31zmu0} by $S_i+S_j-H\id_{\gz_\mu}$ we obtain from \eqref{eq:Gaussinv} $\Phi(\eta_k) \, (S_i+S_j-H\id_{\gz_\mu}) + (\eta_i-\eta_j)(\al_j-\al_i) \, \Psi(\eta_k)\, \id_{\gz_\mu} = 0$. It follows that $\Phi(\eta_k) \Phi(\eta_j) \, (S_i+S_j-H\id_{\gz_\mu}) -\frac14 (\eta_i-\eta_j)^2 \, \Phi(\eta_j)\Psi(\eta_k)\, \id_{\gz_\mu} = 0$. Subtracting the same equation with $j$ and $k$ interchanged we get $\Phi(\eta_k) \Phi(\eta_j) \, (S_j-S_k) - \frac14 (\eta_i-\eta_j)^2 \, \Phi(\eta_j)\Psi(\eta_k)\, \id_{\gz_\mu} +\frac14 (\eta_i-\eta_k)^2 \, \Phi(\eta_k)\Psi(\eta_j)\, \id_{\gz_\mu}= 0$, and so by \eqref{eq:31zmu0},
\begin{equation*}
4\Phi(\eta_k) \Phi(\eta_j)\Phi(\eta_i)- (\eta_j-\eta_k) \, \Psi(\eta_i)\big( (\eta_i-\eta_k)^2 \, \Phi(\eta_k)\Psi(\eta_j)- (\eta_i-\eta_j)^2 \, \Phi(\eta_j)\Psi(\eta_k)\big)=0.
\end{equation*}
The left-hand side is symmetric in $\eta_j, \eta_k$. Using the fact that $\eta_j+\eta_k=-3-\eta_i$ and $\eta_j+\eta_k=\eta_i (\eta_i+3)$, that $p(\eta_i)=0$ and substituting the expression for $\la_1$ from the above we get $2\|V\|^4 + (2-\|V\|^2)^2 + 8\|Y\|^2(1-3\|V\|^2) =0$. But the left-hand side is easily seen to be positive when $\|V\|^2+\|Y\|^2 < 1$, a contradiction. %old 531441*v^10*y^6*(3*v^2 - 2)*(27*v^4*y^2 - 4)(-2*(v^2 + y^2 - 1)*(3*v^4 - 24*v^2*y^2 - 4*v^2 + 8*y^2 + 4))s

This completes the proof of Theorem~\ref{th:noEinDR}.

\end{document}